\documentclass{ifacconf}

\usepackage{xcolor}
\usepackage{graphicx}
\usepackage{natbib}       
\usepackage{amsmath,amssymb}
\usepackage{subcaption}
\usepackage{enumitem}

\usepackage{eso-pic}
\AddToShipoutPictureBG*{%
\AtPageUpperLeft{%
\setlength\unitlength{1in}%
\hspace*{\dimexpr0.5\paperwidth\relax}
\makebox(0,-1.75)[c]{
\begin{tabular}{c c}
Timo de Groot \emph{et al.},
Scaled Graph Bounding Techniques for Reset Systems, \\
To appear in 
{\em 23rd IFAC World Congress},
Busan, South Korea, 2026,
uploaded to ArXiv 20 May 2026 \\
\end{tabular}}}}

\begin{document}
\begin{frontmatter}
\title{Scaled Graph Bounding Techniques for Reset Systems} 
\thanks[footnoteinfo]{The research leading to these results has partially received funding from the European Research Council under the Advanced ERC Grant Agreement PROACTHIS, no. 101055384.}
\author[first]{T. de Groot,} 
\author[first]{W.P.M.H.  Heemels,}
\author[first,second]{T. Oomen,}
\author[first]{S.J.A.M. van den Eijnden} 
\address[first]{Eindhoven University of technology, 
   Eindhoven,  5612 AE Netherlands\\  t.d.groot2@tue.nl, W.P.M.H.Heemels@tue.nl, t.a.e.oomen@tue.nl, s.j.a.m.v.d.eijnden@tue.nl }
   \address[second]{ Delft Center for Systems and Control, Delft University of Technology, Delft, The Netherlands}

\begin{abstract}                
Reset systems can overcome fundamental limitations of linear time-invariant control. The recently introduced notion of scaled (relative) graphs provides a promising framework for developing  graphical analysis and design tools for reset systems, in line with widely adopted loopshaping methods for linear systems. The aim of this paper is to derive techniques for over-bounding the scaled graph of reset systems, and obtain insights in their accuracy. We exploit connections between quadratic dissipativity and scaled graphs to recast the over-bounding problem as the search for piecewise quadratic storage functions. Using specific sampling techniques, we reveal a fundamental limitation of general scaled graph approximation methods that are based on quadratic dissipativity.   
\end{abstract}


\end{frontmatter}

\section{Introduction}
Reset systems, first introduced by \cite{clegg_nonlinear_1958}, can overcome fundamental performance limitations of linear time-invariant (LTI) control \citep{beker_plant_2001,zheng_experimental_2000}. For this reason, reset systems have received considerable interest from both researchers and practitioners over the past decades, see, e.g.,  \citet{zaccarian_first_2005,chait_horowitzs_2002,nesic_stability_2011,aangenent_performance_2010,beker_plant_2001,zheng_experimental_2000,HosseinFrequency}. The adoption of reset controllers in industry is, however, limited, in no small part due to complexity in design and analysis. The nonlinear nature of reset systems limits the use of classical frequency-domain tools such as Bode and Nyquist diagrams \citep{j_macfarlane_multivariable_1976}. The development of intuitive design and analysis tools for reset systems is therefore of great value.

Recently, the notion of scaled relative graphs (SRGs) and scaled graphs (SGs) has been introduced in \cite{ryu_scaled_2022} within the optimization setting, and later brought to the systems and control community in \cite{chaffey_graphical_2023}. In essence, SRGs and SGs provide a graphical representation of the input-output behavior of a system, allowing for intuitive graphical analysis of nonlinear systems in general, and reset systems in particular \citep{van_den_eijnden_scaled_2024,groot_dissipativity_2025,degrootExploitingStructureMIMO2025,krebbekx_reset_2025}. A major challenge in SG analysis comes from obtaining the SG of the systems in question. For LTI systems, exact methods are available \citep{nauta2025computable,krebbekx_graphical_2025,groot_dissipativity_2025}, but for nonlinear systems this is not yet the case, and one should resort to over-bounding techniques.

Recent approaches for over-bounding SGs of reset systems exploit the connection between integral quadratic constraints (IQCs) and SGs, and focus on verifying IQC properties by searching for common quadratic storage functions \citep{van_den_eijnden_scaled_2024, groot_dissipativity_2025}. Though simple and useful, it is known that the use of common quadratic Lyapunov and storage functions can lead to conservatism in analysis, which can be mitigated by resorting to the more versatile class of piecewise quadratic (PWQ) functions \citep{aangenent_performance_2010,johanssonComputationPiecewiseQuadratic1998a,zaccarian_first_2005,nesic_stability_2011}.

The aim of this paper is twofold. We first explore the use of PWQ functions for over-bounding SGs of reset systems. In line with the approach in \cite{groot_dissipativity_2025}, we formulate the over-bounding problem as a convex optimization procedure expressed in terms of linear matrix inequalities (LMIs). The approach allows for a significant reduction of the SG over-bound as compared to methods based on single quadratic functions, as we illustrate in examples for both single-input single-output (SISO) and multi-input multi-output (MIMO) reset systems. Second, we examine the accuracy of the over-bounding techniques. From a finite set of sampled SG points, we derive a lower bound on the achievable over-bounding region. In other words, we identify the smallest over-bounding region our method may be able to find for the system at hand. As we will show, this result reveals that, in certain cases where quadratic dissipativity is exploited, there is no point in going beyond the class of PWQ storage functions (e.g., polynomial functions) to further reduce the over-bounding SG region. 

The remainder of this paper is organized as follows. In Section 2, a background on signals, systems, and SGs is provided and the problem statement is formulated. Next, in Section 3, the over-bounding method based on PWQ storage functions is detailed. Section 4 discusses the accuracy by exploiting a finite number of SG samples. The results are supported by numerical examples throughout. The conclusions are presented in Section 5.

\textbf{Notation:} The fields of real and complex numbers are denoted by $\mathbb{R}$ and $\mathbb{C}$. The sets of non-negative and positive real numbers are denoted $\mathbb{R}_{\geq0}$ and $\mathbb{R}_{>0}$, respectively. The sets of complex numbers with non-negative real part and positive real part are denoted $\mathbb{C}_{\geq0}$ and $\mathbb{C}_{>0}$, while the sets of complex numbers with non-negative imaginary part and non-positive imaginary part are denoted by $\mathbb{C}_+$ and $\mathbb{C}_-$. For any complex number $x\in\mathbb{C}$, we denote $\bar{x}$ as its complex conjugate. For any set $\mathcal{X}\subset \mathbb{C}$, $\textup{conj}(\mathcal{X})$ denotes element wise conjugation i.e. $\textup{conj}(\mathcal{X}):=\{\bar{x}\mid x\in\mathcal{X}\}$. The set of natural numbers (including zero) is denoted $\mathbb{N}$. The set of $n\times n$ symmetric matrices is denoted $\mathbb{S}^n$, and the set of symmetric matrices with positive entries is denoted $\mathbb{S}^n_{\geq0}$. For some $W\in\mathbb{S}^n$, we write $W\succ0$  (resp. $W\succeq0$) to indicate positive definiteness (resp. positive-semidefiniteness). The $n\times n$ identity matrix is written as $I_n$. Vector inequalities are interpreted element wise. A matrix $A\in\mathbb{S}^{n}$ with $a_1,...,a_n$ on its diagonal and $0$ elsewhere is written as $A = \text{diag}(a_1,a_2,...,a_n)$. The Kronecker product is denoted by $\otimes$. 

\section{Background and Problem statement}
\subsection{Reset Systems}
 In this paper, we consider reset systems with state based resetting of the form 
\begin{align}
	\mathfrak{R}: \left\{\begin{aligned}
		\dot{x}(t)  &= Ax(t)+Bu(t), & \textnormal{if}\; x(t)\in\mathcal{F},\\
		x(t^+) &=Rx(t), & \textnormal{if}\;x(t)\in \mathcal{J},\\
		y(t) &=Cx(t)+Du(t), &
	\end{aligned} \right. \label{eq:canonical_reset_system}
\end{align}
with states $x(t)\in\mathbb{R}^n$, initial condition $x(0)=0$, input $u(t)\in\mathbb{R}^p$, output $y(t)\in \mathbb{R}^p$ all at time $t \in \mathbb{R}_{\geq 0}$, and where $x(t^+) = \lim_{s \downarrow t}x(s)$. The matrix $R \in \mathbb{R}^{n\times n}$ defines the reset map, and the jump and flow sets are given by
\begin{align}
	\mathcal{F} &= \{\mu \in \mathbb{R}^{n} \;|\;\mu^\top M\mu\geq 0\} \label{eq:canonical_flow_set_reset},\\
	\mathcal{J} &= \{\mu \in \mathbb{R}^{n} \;|\; \mu^\top M\mu<0\}\label{eq:canonical_jump_set_reset},
\end{align}
with $M\in\mathbb{S}^{n}$. We write $y\in\mathfrak{R} u$ to denote the (possibly multi-valued) output $y$ as a consequence of applying the input $u$ to the reset system $\mathfrak{R}$ in \eqref{eq:canonical_reset_system}. 

Let the space of square-integrable signals be given by
\begin{align} 
\mathcal{L}_2:=\left\{u:\mathbb{R}_{\geq0}\rightarrow\mathbb{R}^p \mid \int_0^\infty u^\top(t)u(t) dt < \infty \right\}.
\end{align} 
For $u,y\in \mathcal{L}_2$ we define the inner-product $\langle u,y\rangle=\int_0^\infty u^\top\!(t)y(t)\;dt$ and norm $\|u\| = \sqrt{\langle u,u\rangle}$.

We make the following standing assumptions.
\begin{assum}\label{ass:stable_state}
The reset system in \eqref{eq:canonical_reset_system} with input $u\in\mathcal{L}_2$: 
\begin{enumerate}[label=\alph*)]
    \item exhibits no Zeno behavior;
    \item   satisfies for $u\in\mathcal{L}_2$, $x,y\in\mathcal{L}_2$  and $\lim_{t\to \infty} x(t) = 0$.
\end{enumerate}
\hfill $\square$
\end{assum}
Item a) in Assumption~\ref{ass:stable_state} removes the possibility of an infinite number of resets in a finite time interval. Item b) imposes a stability property, which holds if the matrix $A$ in \eqref{eq:canonical_reset_system} is Hurwitz. 

\subsection{Scaled Graphs}
For a reset system $\mathfrak{R}$ in \eqref{eq:canonical_reset_system} and some input $u\in\mathcal{L}_2$, we define the set of complex numbers
\begin{align} \label{eq:zu}
z(u,\mathfrak{R}) = \left\{ \left.\rho(u,y)e^{\pm j\theta(u,y)} \right|\;  \:y\in\mathfrak{R}u   \right\},
\end{align} 
with $\rho(u,y) = \|y\|/\|u\|$,  and $\theta(u,y)=\arccos\left(\tfrac{\langle u,y\rangle}{\|u\|\|y\|}\right)$ if $y\neq 0$, and $\theta(u,y)=0$ if $y=0$. The SG of $\mathfrak{R}$ in \eqref{eq:canonical_reset_system} is then defined as 
\begin{align}\label{eq:SG}
    \textup{SG}(\mathfrak{R}) =  \left\{ z(u,\mathfrak{R}) \mid u \in \mathcal{L}_2 \setminus \{0\} \right\}.
\end{align}
The power of SGs comes from the fact that they can be used to graphically check stability of a feedback loop containing nonlinear systems, similar to the use of Nyquist diagrams to verify stability of LTI feedback loops.

\subsection{Dissipativity and Scaled Graphs}
Next, we recall the main connection between dissipativity, IQCs, and SGs for reset systems from \cite{groot_dissipativity_2025}.
\begin{lem}\label{lem:connection_to_SG} Consider a symmetric matrix $\Pi\in \mathbb{S}^2$. Then, the reset system $\mathfrak{R}$ in \eqref{eq:canonical_reset_system} satisfies the IQC 
\begin{align} \label{eq:Supply_inequality}
\int_{0}^{\infty}\begin{bmatrix} y(t) \\ u(t) \end{bmatrix}^\top \left(\Pi \otimes I_p\right)\begin{bmatrix} y(t) \\ u(t) \end{bmatrix} dt \geq 0,
\end{align} 
for all $u\in \mathcal{L}_2$ and $y \in \mathfrak{R}u$ if and only if $\textup{SG}(\mathfrak{R}) \subset \mathcal{S}(\Pi)$, in which
\begin{align} 
	\mathcal{S}(\Pi) = \left\{ z \in\mathbb{C} \left|  \begin{bmatrix}
		z \\ 1
	\end{bmatrix}^* \Pi \begin{bmatrix}
		z \\ 1
	\end{bmatrix} \geq 0 \right.  \right\}.\label{eq:S_PI_canonical}
\end{align}
\hfill $\square$
\end{lem}
Let the set $\mathbf{\Pi}(\mathfrak{R})$ be given by
$$\mathbf{\Pi}(\mathfrak{R}):=\{\Pi\in\mathbb{S}^2 \mid  \text{ \eqref{eq:Supply_inequality} holds for all $u\in\mathcal{L}_2,y\in \mathfrak{R} u$} \}$$ and let  $\tilde{\mathbf{\Pi}} \subset \mathbf{\Pi}(\mathfrak{R})$. Then it follows directly from Lemma \ref{lem:connection_to_SG} that
\begin{align} \label{eq:defn_K}
\textup{SG}(\mathfrak{R})\subset \bigcap_{\Pi\in\tilde{\mathbf{\Pi}}}\mathcal{S}(\Pi)=:\mathcal{S}_{\textup{SG}}.
\end{align} 

The parameterization of $\Pi$ as
\begin{align} \label{eq:pi_parameterization}
\Pi(\sigma,\lambda_c,r) = \sigma\begin{bmatrix} 1 & -\lambda_c\\-\lambda_c & \lambda_c^2 -r^2  \end{bmatrix}  
\end{align} 
where $\sigma\in\{-1,+1\}$, $\lambda_c\in\mathbb{R}$ and $r\in\mathbb{R}_{\geq0}$, is of specific interest, as this leads to $\mathcal{S}(\Pi(-1,\lambda_c,r))$ and $\mathcal{S}(\Pi(+1,\lambda_c,r))$ representing the interior and exterior of a disc centered on the real line at $\lambda_c$ and with radius $r$, respectively, see \cite{groot_dissipativity_2025} for more details. 

\subsection{Problem Statement}

 In this paper, we are interested in constructing tight over-bounds of $\textup{SG}(\mathfrak{R})$, with $\mathfrak{R}$ a reset system as in \eqref{eq:canonical_reset_system}. Specifically, we are interested in the following aspects:
\begin{enumerate}
    \item Obtaining tight over-bounds of $\textup{SG}(\mathfrak{R})$ by verifying quadratic dissipativity via flexible storage functions; 
    \item Characterization of the smallest achievable over-bound via quadratic dissipativity.
\end{enumerate}

For 1) we exploit the class of PWQ storage functions to make significant refinements  compared to the use of quadratic storage functions as in, e.g., \cite{groot_dissipativity_2025} (as we demonstrate in an example).

For 2) we use a specific sampling technique in combination with the hyperbolic convex hull to construct the \textit{minimal} achievable over-bound via quadratic dissipativity.

We will start with item 1) in Section 3 and subsequently develop item 2) in Section 4.

\section{Over-bounding scaled graphs}
\subsection{Piecewise Quadratic Storage Functions}
A standard method to verify the IQC in \eqref{eq:Supply_inequality} for reset systems is via the use of PWQ storage functions of the form $W(x) = x^\top P_i x$ with $P_i\in\mathbb{S}^{n}$ and $i\in\{1,\dots,N\}$. For this purpose, the flow and jump sets are partitioned, and a $P_i$ is associated with each partition.

In this paper, we consider a polytopic partitioning of the jump and flow sets denoted $\mathcal{X}:=\bigcup_{i=1}^N\left\{\mathbb{X}_i\right\}$, with $N\in\mathbb{N}$. Each partition is described by
\begin{align*} 
\mathbb{X}_i :=\{x\in\mathbb{R}^n\mid E_ix\geq0\} \; \text{for}\; i\in\{1,\dots,N\},\; E_i\in\mathbb{R}^{n\times n}.
\end{align*} 
We define $\mathcal{M}$ to be the set of indices corresponding to the partitions of the flow set such that $\bigcup_{i\in\mathcal{M}} \mathbb{X}_i  =\mathcal{F}$. Similarly, we define $\mathcal{N}$ to be the set of indices corresponding to the partitions of the jump set, such that $\bigcup_{j\in\mathcal{N}} \mathbb{X}_j =\mathcal{J}$. Note that by construction $\mathcal{M}\cup\mathcal{N}=\{1,\dots,N\}$.

To enforce continuity of the storage function, we employ the parameterization $P_i = F_i^\top \Phi F_i$,  where $\Phi\in\mathbb{R}^{N\times N}$, and $F_i\in\mathbb{R}^{N\times n}$ has the property that 
\begin{align*} 
F_kx = F_lx \quad \text{if}\; x\in\mathbb{X}_k\cap \mathbb{X}_l \quad \text{for all } k,l\in\mathcal{M}\cup\mathcal{N},\;k\neq l.
\end{align*}
The matrices $F_i$ are referred to as continuity matrices. In general, the continuity matrices are not unique, and their construction depends on the type of polytopic partition used \citep{johanssonComputationPiecewiseQuadratic1998a}. See also \cite{erlandsen_continuity_2026}.

\subsection{Main Result}
Next, we present our main result for the over-bound of $\textup{SG}(\mathfrak{R})$. 

\begin{thm}\label{thm:main_reset_PWQ_LMIs}
Consider a reset system $\mathfrak{R}:\mathcal{L}_2\rightrightarrows\mathcal{L}_2$ as in \eqref{eq:canonical_reset_system}. Let a partitioning $\mathcal{X}$ and the corresponding continuity matrices $E_k, F_k$ with $k\in \mathcal{M}\cup \mathcal{N}$ be given. If, for all $i\in\mathcal{M}$ and all $j \in \mathcal{N}$, there exist $\Pi \in\mathbb{S}^{2}$, $\Phi \in \mathbb{S}^{N}$, $U_{1,i},U_{2,i},U_{3,j},U_{4,j}\in\mathbb{S}_{\geq0}^{n}$ such that the following LMIs are satisfied
    \begin{align} 
    \begin{bmatrix}
			A^\top\! P_i + P_iA & P_iB\\B^\top P_i&0
		\end{bmatrix} \!-\!\Theta(\Pi) \!+\! \text{diag}(E_i^\top U_{1,i}E_i,0) & \!\preceq 0, \label{eq:lmi1}\\
    R^\top P_{j}R-P_{j} + E_{j}^\top U_{3,j}E_{j} &\preceq 0, \label{eq:lmi2}\\
    R^\top P_{i}R-P_{j} + E_{j}^\top U_{4,j}E_{j}+R^\top E_i^\top U_{2,i}E_iR & \!\preceq 0, \label{eq:lmi3}
    \end{align} 
    with
    \begin{align} 
        \Theta(\Pi) &= \begin{bmatrix}
		C & D \\ 0 & I
	\end{bmatrix}^\top(\Pi\otimes I_n)\begin{bmatrix}
		C & D \\ 0 & I
	\end{bmatrix},\\
    P_i &= F_i^\top \Phi F_i,\label{eq:cont_eq_1}\\
    P_{j} &= F_{j}^\top \Phi F_{j},\label{eq:cont_eq_2}
    \end{align} 
    then it holds that
    \begin{align} \label{eq:thm_inclusion}
\textup{SG}(\mathfrak{R})\subset \mathcal{S}(\Pi).
\end{align} 
\hfill $\square$
\end{thm}
The proof follows along similar lines as the proofs in \cite{aangenent_performance_2010} and \cite{van_den_eijnden_scaled_2024}.
\begin{rem}
   The LMI in \eqref{eq:lmi1} encodes the dissipativity inequality during flow, where we have used the S-procedure relaxation. The LMIs in \eqref{eq:lmi2} and \eqref{eq:lmi3} ensure that the storage function decreases during jumps, where the S-procedure relaxation is also used. The equalities in \eqref{eq:cont_eq_1} and \eqref{eq:cont_eq_2} ensure continuity of the storage function. \hfill $\square$
\end{rem}
\begin{rem}
        Theorem \ref{thm:main_reset_PWQ_LMIs} can be extended to apply to hard SGs, introduced in \cite{chen2025softhardscaledrelative}, which are useful for working with unbounded systems such as integrators. Following the reasoning of Section 6 of \cite{groot_dissipativity_2025}, constraints $P_k\succeq 0 $ for all $k\in\mathcal{M}\cup\mathcal{N}$ should be added to the LMIs in \eqref{eq:lmi1}-\eqref{eq:lmi3}, so that the inclusion in \eqref{eq:thm_inclusion} acts on the hard SG. \hfill $\square$
\end{rem}

Using Theorem \ref{thm:main_reset_PWQ_LMIs} we can construct an over-bounding of $\textup{SG}(\mathfrak{R})$. To this end, we consider the parameterization of $\Pi$ in \eqref{eq:pi_parameterization} and collect the left hand side of the LMI in (\ref{eq:lmi1}),\eqref{eq:lmi2} and (\ref{eq:lmi3}) in the LMI $L(\sigma,\!\lambda_c,\!r,\!\Phi,\!U_{1,i},\!U_{2,i},\!U_{3,j},\!U_{4,j})$. Furthermore, we define two sets $\Lambda_i$ and $\Lambda_e$, which contain the center points for the circular regions given by $\mathcal{S}(\Pi)$. The corresponding radii are then found by solving the following optimization problem \textbf{Prob}.

\textbf{Prob:} For each $(\sigma,\lambda_c) \in \{-1\}\times\Lambda_i \cup \{1\}\times \Lambda_e$, 
\begin{equation*}
\begin{aligned}
&\underset{r,\Phi, U_{1,i},U_{2,i},U_{3,j},U_{4,j}}{\makebox[0pt][r]{\text{minimize}}}        & & \sigma r^2\\
&\text{subject to} & &L(\sigma,\lambda_c,r,\Phi,U_{1,i},U_{2,i},U_{3,j},U_{4,j})\preceq 0, \\
& & & \Phi\in\mathbb{S}^N, U_{1,i},U_{2,i},U_{3,j},U_{4,j} \in \mathbb{S}_{\geq0}^n,\\ &&& i\in\mathcal{M}, j\in\mathcal{N},r>0.
\end{aligned}
\end{equation*}

 All $\Pi$ matrices corresponding to the solutions of the optimization problem \textbf{Prob} are collected in the set $\tilde{\Pi}$.  The over-bound is then constructed by invoking Lemma \ref{lem:connection_to_SG}.

\subsection{Examples}
We illustrate the effectiveness of Theorem \ref{thm:main_reset_PWQ_LMIs} via examples of SISO and MIMO reset systems. 

\begin{exmp}[SISO]
Consider the SISO reset system denoted $\mathfrak{R}_1$, described by \eqref{eq:canonical_reset_system} with the data
\begin{align}\label{eq:R1}
		\left[\begin{array}{c|c }
			A&B   \\
			\hline 
			C&D		
		\end{array}\right] = 
		\left[\begin{array}{c c|c }
			 -1 & 0 & 1   \\
			 1 & -1 & 0\\
			\hline 
			0 & 1 & 0
		\end{array}\right],\;\begin{aligned}
		    R&=\textup{diag}(0,0),\\ M&=\textup{diag}(0.9^2,-1).
		\end{aligned}
\end{align}

As a partition strategy, we consider polytopic conical partitions of the jump and flow sets. The flow set is subdivided into $N-2$ polytopic cones, with equidistant angles. The jump set can be described by two polytopic cones and is not subdivided further (see also \cite{aangenent_performance_2010}).

Having constructed the cell bounding matrices and the continuity matrices, we solve the optimization problem \textbf{Prob} with $N=40$, $\Lambda_i=\{-1+0.05k \mid k = 0,1, \dots, 80\}$ and $\Lambda_e=\{-1+0.25k \mid k = 0,1, \dots, 80\}$ and construct the over-bound using Lemma \ref{lem:connection_to_SG}. The result is shown by the blue region in Fig. \ref{fig:reset_comparison}.

 Fig. \ref{fig:reset_comparison} shows that the over-bound of $\textup{SG}(\mathfrak{R}_1)$ constructed using Theorem \ref{thm:main_reset_PWQ_LMIs} is significantly tighter than the over-bound found using common quadratic storage functions. This shows the improvements that more flexible PWQ storage functions offer over common quadratic ones.

\begin{figure}[t]
    \centering
    \includegraphics[width=0.38\linewidth]{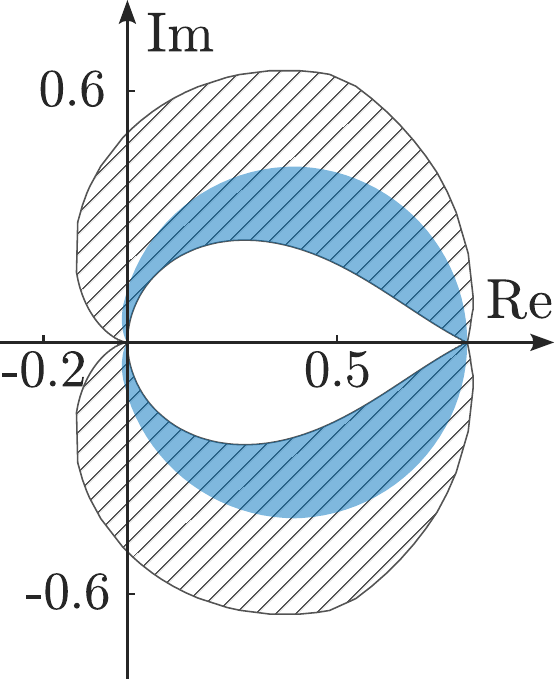}
    \caption{Over-bound of $\textup{SG}(\mathfrak{R}_1)$ using PWQ storage functions 
 (blue), and using common quadratic storage functions
 (hatched). The over-bound using PWQ storage functions is significantly smaller and thus less conservative, compared to the bound obtained using common quadratic storage functions.}
    \label{fig:reset_comparison}
\end{figure}
    
\end{exmp}

\begin{exmp}[MIMO]
Consider a MIMO reset system described by (\ref{eq:canonical_reset_system}) with data
\begin{align*}
		\left[\begin{array}{c|c }
			A&B   \\
			\hline 
			C&D		
		\end{array}\right] \!=\! 
		\left[\begin{array}{c c|c c }
			 -1 & 0 & 1 & 0  \\
			 1 & -1 & 0 & 1\\
			\hline 
			1 & 0 & 0 & 0\\
            0 & 1 & 0 & 0
		\end{array}\right]\;\begin{aligned}
		    R&=\textup{diag}(0,0),\\ M&=\textup{diag}(0.9^2,-1).
		\end{aligned}
\end{align*}

which is denoted $\mathfrak{R}_2$. We partition the state-space in the same way as for $\mathfrak{R}_1$, with $N=40$, and $\Lambda_i$ $\Lambda_e$ as before. The resulting over-bound of $\textup{SG}(\mathfrak{R}_2)$ using Theorem \ref{thm:main_reset_PWQ_LMIs}, as well as the over-bound based on common quadratic storage functions, are shown in Fig. \ref{fig:reset_MIMO}.

\begin{figure}[b]
\centering
    \includegraphics[width=0.35\linewidth]{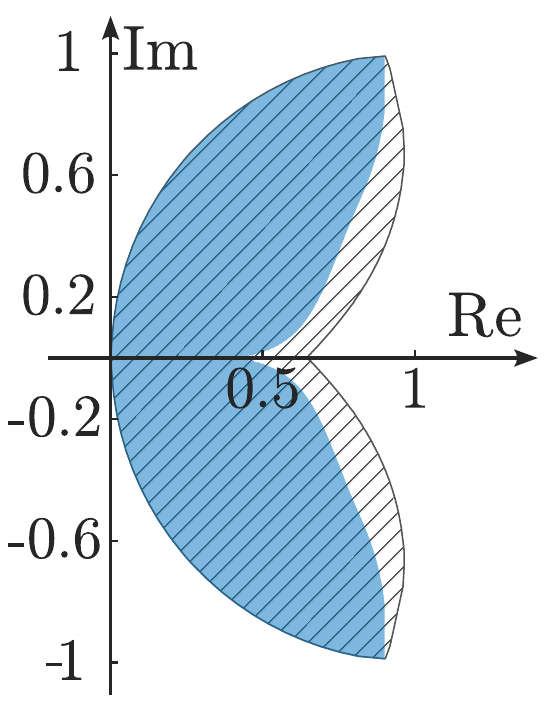}
\caption{ Over-bound of $\textup{SG}(\mathfrak{R}_2)$ using PWQ storage functions 
 (blue), and using common quadratic storage functions
 (hatched).}
    \label{fig:reset_MIMO}
\end{figure} 

Fig. \ref{fig:reset_MIMO} shows that the over approximation based on a PWQ storage function is smaller than the respective over-bound based on a common quadratic storage function.    
\end{exmp}

\section{Minimal Achievable Over-Bounds}
In the previous section, we have shown that using PWQ storage functions can significantly tighten the over-bound of reset system SGs. The question remains whether the conservatism  could be reduced further by showing quadratic dissipativity using different classes of storage functions.

 In the first part of this section we provide some useful machinery from hyperbolic geometry, which will play a key role in our second main result. In the second part, we state our second main result in Theorem \ref{thm:samples}, which provides a minimal achievable over-bound of SGs using quadratic dissipativity. In the third part, we apply this result to the example system $\mathfrak{R}_1$ To illustrate the usefulness of the result.

 \subsection{Hyperbolic Geometry} \label{sec:hyperbolic_geometry}
 The geodesic between two points $z_1,z_2\in\mathbb{C}_{+}$ under the Poincaré metric is denoted $\text{ARC}_{\text{min}}(z_1,z_2)$. This is the section of the circle in $\mathbb{C}_+$, centered on the real axis and having endpoints $z_1$ and $z_2$. We define hyperbolic convexity (h-convexity) and the h-convex hull (h-hull) as follows.
\begin{defn}\textnormal{\cite[Definition 6]{chaffey_graphical_2023}} A set $\mathcal{C}\subset\mathbb{C}_{+}$ is said to be h-convex if
	\begin{align} 
		z_1,z_2\in \mathcal{C} \Rightarrow  \textup{ARC}_{\textup{min}}(z_1,z_2)\in \mathcal{C}.
		\end{align} 
        The h-convex hull of $\mathcal{C}$, denoted by h-hull($\mathcal{C}$), is the smallest h-convex set containing $\mathcal{C}$. \hfill $\square$
\end{defn}
We collect those subsets of $\mathbb{C}$ that are symmetric around the real axis in the set
 \begin{align} 
 \mathbb{C}\mathbb{S} :=\{ \mathcal{C}\subset\mathbb{C} \mid \mathcal{C}\cap\mathbb{C}_+ = \textup{conj}(\mathcal{C}\cap\mathbb{C}_-) \}
\end{align} 
Next, we define the h-convex hull for sets $\mathcal{C}\in\mathbb{C}\mathbb{S}$ as 
\begin{align} 
	\textup{h-hull}(\mathcal{C})\!\! :=\! \textup{h-hull}(\mathcal{C}\cap\mathbb{C}_+) \!\cup \textup{conj}(\textup{h-hull}(\mathcal{C}\cap\mathbb{C}_+)).
\end{align} 

Examples of $\textup{h-hull}(\mathcal{C})$ are shown in Fig. \ref{fig:h-hull-examples} for both cases $\mathcal{C}\subset\mathbb{C}_+$ and $\mathcal{C}\in\mathbb{C}\mathbb{S}$. Note that, by definition, the h-hull preserves set inclusion, that is, for any sets $\mathcal{X}, \mathcal{Y} \subset \mathbb{C}\mathbb{S}$, if $\mathcal{X} \subset \mathcal{Y}$, then
$\operatorname{h\text{-}hull}(\mathcal{X}) \subset \operatorname{h\text{-}hull}(\mathcal{Y})$.

\subsection{Minimal achievable over-bounds}
Consider the subset of inputs $\mathcal{U}\subset\mathcal{L}_2$. We define a set of samples of $\textup{SG}(\mathfrak{R})$ as
\begin{align} 
	\mathcal{Z}(\mathfrak{R},\mathcal{U}) := \{ z(u,\mathfrak{R}) \mid u\in\mathcal{U} \},
\end{align}
with $z(u,\mathfrak{R})$ in \eqref{eq:zu}, such that, by definition, $ \mathcal{Z}(\mathfrak{R},\mathcal{U})\subset \textup{SG}(\mathfrak{R}) $. We now state our next result.
\begin{thm}\label{thm:samples}
	Given a reset system $\mathfrak{R}:\mathcal{L}_2\rightrightarrows\mathcal{L}_2$, an over-bound of $\textup{SG}(\mathfrak{R})$, denoted $\mathcal{S}_{\textup{SG}}$ from  \eqref{eq:defn_K}, and a set of samples $\mathcal{Z}(\mathfrak{R},\mathcal{U})$. Then for any $\mathcal{U}\subset \mathcal{L}_2$
	\begin{align} 
		\textup{h-hull}(\mathcal{Z}(\mathfrak{R},\mathcal{U})) \subset \mathcal{S}_{\textup{SG}}. \label{eq:lower_bound_approx}
\end{align} 
\end{thm}
\begin{pf}
First, note that $\mathcal{S}_{\textup{SG}}$ is symmetric around the real axis. Then, by the definition of $\mathcal{S}_{\textup{SG}}$ in \eqref{eq:defn_K} we get
\begin{align} 
\mathcal{S}_{\textup{SG}} \cap \mathbb{C}_+ = \left( \bigcap_{\Pi\in\tilde{\mathbf{\Pi}}}\mathcal{S}(\Pi) \right) \cap \mathbb{C}_+.
\end{align} 
This shows that $\mathcal{S}_{\textup{SG}}\cap\mathbb{C}_+$ is the intersection of hyperbolic half-planes, therefore $\mathcal{S}_{\textup{SG}}=\textup{h-hull}(\mathcal{S}_{\textup{SG}})$. Next, from the definition of $\mathcal{Z}(\mathfrak{R},\mathcal{U})$ we get $\mathcal{Z}(\mathfrak{R},\mathcal{U})\subset \textup{SG}(\mathfrak{R})$, and from Lemma \ref{lem:connection_to_SG} we get $\textup{SG}(\mathfrak{R})\subset \mathcal{S}_{\textup{SG}}$. Therefore, $\mathcal{Z}(\mathfrak{R},\mathcal{U})\subset\mathcal{S}_{\textup{SG}}$. Taking the hyperbolic convex-hull on both sides gives
\begin{align} 
\textup{h-hull}(\mathcal{Z}(\mathfrak{R},\mathcal{U})) \subset  \textup{h-hull}(\mathcal{S}_{\textup{SG}})=\mathcal{S}_{\textup{SG}}.
\end{align} 
This concludes the proof. \hfill $\square$
\end{pf}
Note that in general $\textup{h-hull}(\mathcal{Z}(\mathfrak{R},\mathcal{U}))\nsubseteq \textup{SG}(\mathfrak{R})$.

\begin{figure}[t]
\centering
    \includegraphics[width=3.3cm]{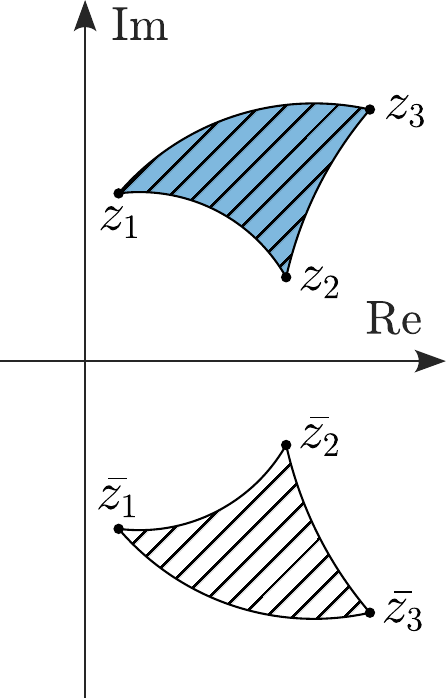}
\caption{ Examples of h-hulls of $\mathcal{C} = \{z_1,z_2,z_3\}\subset\mathbb{C}_+$ (blue)  and $\mathcal{C} =  \{z_1,\bar{z}_1,z_2,\bar{z}_2,z_3,\bar{z}_3\}\in\mathbb{C}\mathbb{S}$ (hatched).}
    \label{fig:h-hull-examples}
\end{figure} 

Theorem \ref{thm:samples} shows that $\textup{h-hull}(\mathcal{Z}(\mathfrak{R},\mathcal{U}))$ acts as a minimal achievable over-bound for $\textup{SG}(\mathfrak{R})$. We formalize this in the following corollary.
\begin{cor}\label{cor:limits}
    If $\textup{h-hull}(\mathcal{Z}(\mathfrak{R},\mathcal{U}))\!=\!\mathcal{S}_{\textup{SG}}$, then $\mathcal{S}_{\textup{SG}}$ is the minimal achievable over-bound of $\textup{SG}(\mathfrak{R})$ using quadratic dissipativity.
\end{cor}
\begin{pf}
    The proof follows directly from Theorem \ref{thm:samples} and the fact that $\mathcal{S}_{\textup{SG}}$ is constructed using intersections of circular regions which are found using quadratic dissipativity. \hfill $\square$
\end{pf}
\begin{rem}
Although, we limit ourselves to reset systems in this paper, Theorem \ref{thm:main_reset_PWQ_LMIs} and Corollary \ref{cor:limits} immediately extend to SGs of other classes of systems for which quadratic dissipativity is used to calculate over-bounds, such as piecewise linear systems \citep{groot_dissipativity_2025} and Lur'e systems  \citep{degrootExploitingStructureMIMO2025}. \hfill $\square$
\end{rem}

Since $\mathcal{Z}(\mathfrak{R},\mathcal{U})\subset \mathcal{S}_{\textup{SG}}$, Corollary \ref{cor:limits} can be formulated using $\mathcal{Z}(\mathfrak{R},\mathcal{U})$, without the need for the h-hull. However, Using the h-hull provides significant benefits over using the samples directly. In general, since $\textup{h-hull}(\mathcal{Z}(\mathfrak{R},\mathcal{U}))\nsubseteq\textup{SG}(\mathfrak{R})$,  $\textup{h-hull}(\mathcal{Z}(\mathfrak{R},\mathcal{U}))$ can cover regions where no samples can exist as they are not part of $\textup{SG}(\mathfrak{R})$. Furthermore, $\textup{h-hull}(\mathcal{Z}(\mathfrak{R},\mathcal{U}))$ can cover a large part of $\textup{SG}(\mathfrak{R})$ using only a few samples, so fewer samples are needed to construct a minimal achievable over-bound. The usefulness of Theorem \ref{thm:samples} is illustrated in the example in the next subsection.

\begin{figure}[b]
    \centering
    \includegraphics[width=0.5\linewidth]{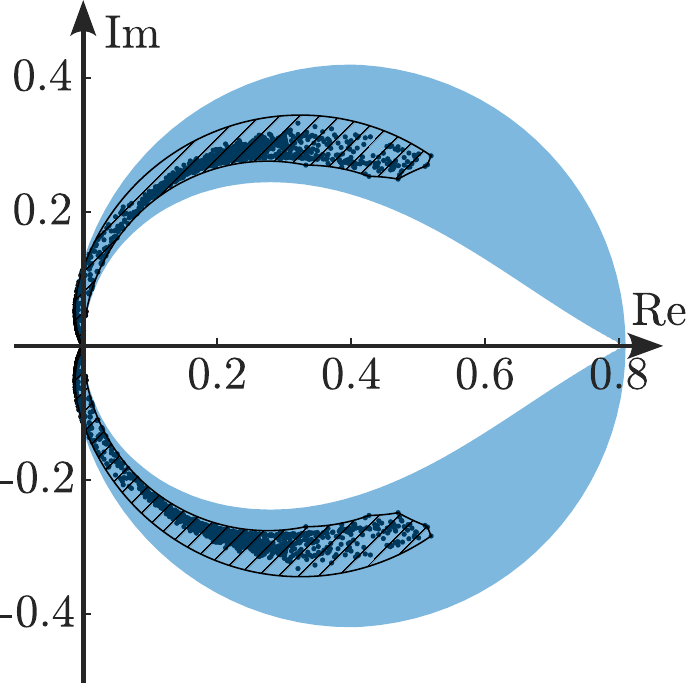}
    \caption{Over-bound $\mathcal{S}_{\textup{SG},1}$ of $\textup{SG}(\mathfrak{R}_1)$ (blue region), $\mathcal{Z}_{\text{ms}}$ (black points), and $\textup{h-hull}(\mathcal{Z}_{\text{ms}})$ (black hatched region). $\mathcal{Z}_{\text{ms}}$ and $\textup{h-hull}(\mathcal{Z}_{\text{ms}})$ partially cover $\mathcal{S}_{\textup{SG},1}$. }
    \label{fig:sampling_multisin_hhul}
\end{figure}

\subsection{Example 1 revisited}
We again consider the example reset system $\mathfrak{R}_1$ in \eqref{eq:R1} and seek samples in order to apply Theorem \ref{thm:samples} and obtain insight into the possible conservatism in the over-bound of $\textup{SG}(\mathfrak{R}_1)$ denoted $\mathcal{S}_{\textup{SG},1}$. To this end, we consider two sets of input signals: 1) multi-sines and 2) switching sigmoids.

Consider the finite set of random multi-sine signals $\mathcal{U}_{ms}$, the resulting samples $\mathcal{Z}_{ms}=\mathcal{Z}(\mathfrak{R}_1,\mathcal{U}_{ms})$, and $\mathcal{S}_{\textup{SG},1}$. We compare $\textup{h-hull}(\mathcal{Z}_{ms})$ with $\mathcal{S}_{\textup{SG},1}$, which are both, together with $\mathcal{Z}_{ms}$, shown in Fig. \ref{fig:sampling_multisin_hhul}. It can be seen that there are significant regions of $\mathcal{S}_{\textup{SG},1}$ that are not covered by $\textup{h-hull}(\mathcal{Z}_{ms})$. This can indicate two scenarios: 1) there is significant conservatism in the over-bound and extending the class of storage functions considered could potentially improve the over-bound; 2) considering a wider class of inputs can increase the region covered by $\textup{h-hull}(\mathcal{Z}(\mathfrak{R},\mathcal{U}))$. We investigate the second case.

We are interested in sampling the SG around the point $z=0.81$, as this is the region where samples are missing in Fig. \ref{fig:sampling_multisin_hhul}. For any $z\in\textup{SG}(\mathfrak{R}_1)\cap\mathbb{R}$ it must hold that $\langle u,y\rangle = \|u\|\|y\|$, i.e. the phase component of the SG is zero for these inputs. This inequality holds if $\mathfrak{R}_1$ behaves as a gain for an input or set of inputs, i.e., $y=\alpha u$, with $\alpha\in\mathbb{R}$. 

In order to achieve ``gain-like" behavior, we consider discontinuous input signals. Specifically, signals that reset at the same time as the states are reset, such that the system only needs to behave as a gain during flow. To describe the resetting of the input signal, we construct a timer $\tau$, which is given by the solution to
\begin{align} 
\begin{cases}
    \dot{\tau}(t)= 1 &\text{if}\; \xi(t) \in \mathcal{F},\\
    \tau(t^+) = 0 & \text{if}\; \xi(t) \in \mathcal{J}.
\end{cases} 
\end{align}  
The resetting input signal can then be written as $u(\tau(t))$. 

\begin{rem}
    Constructing the input signal in this manner creates an indirect feedback loop (in the sense that the input depends on the state). The input signal can be regarded as a signal in $\mathcal{L}_2$, and thus results, when used for sampling, in a point in $\textup{SG}(\mathfrak{R}_1)$. From another point of view, consider an input signal $\bar{u}(t) = u(t)$ for all $t\in \mathbb{R}_{\geq0}$, but which does not directly depend on jumps of the state of $\mathfrak{R}_1$. Then, $\mathfrak{R}_1\bar u = \mathfrak{R}_1 u$ and thus the sampled point will be identical. \hfill $\square$
\end{rem}

For the resetting input signal class, we consider exponential decaying sigmoids, given by
\begin{align} 
u(t) = \frac{e^{-\mu t}}{1+e^{-at+\nu}},
\end{align} 
with $\mu,a,\nu\in\mathbb{R}_{>0}$, such that $u\in\mathcal{L}_2$. The intuition for this class of inputs is as follows. Sigmoids are monotonic and have a dominant low-frequency content. Due to the low-pass behavior of the flow dynamics of $\mathfrak{R}_1$, this results in gain-like behavior, especially for slower rising sigmoids.

Using the sigmoid function and the reset mechanism, we define the set of inputs
\begin{multline*} 
\mathcal{U}_{\text{sig}}:=\left\{u(\tau(t)) \mid \right. \\ \left. t\in\mathbb{R}_{\geq0} , \mu=0.001,a\in [0.05,10] \nu\in[0.1,100]  \right\}.
\end{multline*} 

The resulting sample points $\mathcal{Z}_{\text{sig}}:=\mathcal{Z}(\mathfrak{R}_1,\mathcal{U}_{\text{sig}})$ and $\mathcal{S}_{\textup{SG},1}$ are shown in Fig. \ref{fig:reset_resampled}. It can be seen that $\mathcal{Z}_{\text{sig}}$ fills a significant part of $\mathcal{S}_{\textup{SG},1}$ around $z=0.81$, showing that this region is indeed part of $\textup{SG}(\mathfrak{R}_1)$.

In order to investigate the conservatism of $\mathcal{S}_{\textup{SG},1}$, we define $\mathcal{Z}_{1}:=\mathcal{Z}_{\text{sig}}\cup \mathcal{Z}_{\text{ms}}$. Using Theorem \ref{thm:samples} we find that $\textup{h-hull}(\mathcal{Z}_1)\subset\mathcal{S}_{\textup{SG},1}$. In Fig. \ref{fig:reset_resampled} $\textup{h-hull}(\mathcal{Z}_1)$ is shown as the black hatched region. Fig.\ref{fig:reset_resampled} shows that $\textup{h-hull}(\mathcal{Z}_1)$ covers a large part of $\mathcal{S}_{\textup{SG},1}$ including the outer boundary of $\mathcal{S}_{\textup{SG},1}$.  The conclusion from this result is that the outer boundary of $\textup{SG}(\mathfrak{R}_1)$ cannot be more tightly over-bounded than the outer boundary of $\mathcal{S}_{\textup{SG},1}$, by using other storage functions to verify quadratic dissipativity.

Note that in order to reach the same conclusion using sample points directly (i.e., without the h-hull), sample points must be found that lie along the outer boundary of $\mathcal{S}_{\textup{SG},1}$. However, such sample points might not exist as $\textup{h-hull}(\mathcal{Z}_1)\nsubseteq \textup{SG}(\mathfrak{R}_1)$.  This illustrates the benefit of using $\textup{h-hull}(\mathcal{Z}_1)$ over $\mathcal{Z}_1$ directly.

\begin{figure}[b]
    \centering
    \includegraphics[width=8.4cm]{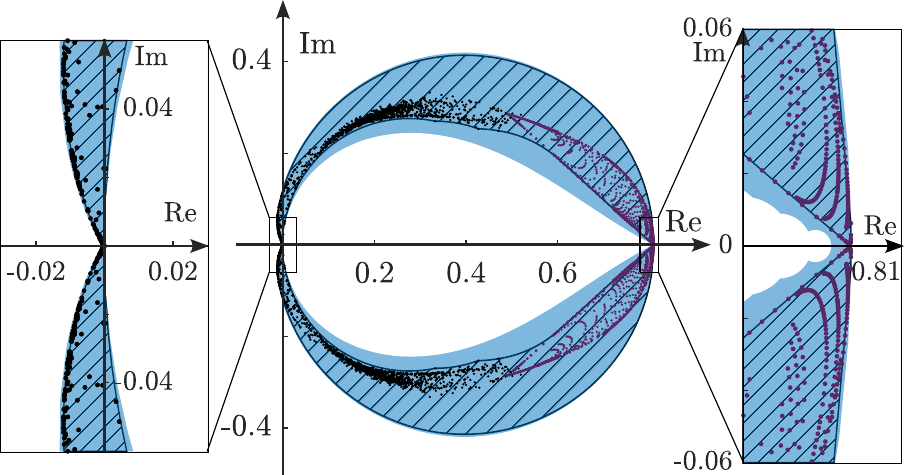}
    \caption{ Over approximation $\mathcal{S}_{\textup{SG},1}$ of $\textup{SG}(\mathfrak{R}_1)$ (blue region), $\mathcal{Z}_{\text{ms}}$ (black points), $\mathcal{Z}_{\text{sig}}$ (purple points) and $\textup{h-hull}(\mathcal{Z}_1)$ (black hatched region). It can be seen that $\textup{h-hull}(\mathcal{Z}_1)$ and $\mathcal{S}_{\textup{SG},1}$ significantly overlap and that the outer boundary of $\textup{h-hull}(\mathcal{Z}_1)$ is close to the outer boundary of $\mathcal{S}_{\textup{SG},1}$. Therefore,  the outer boundary of $\textup{SG}(\mathfrak{R}_1)$ cannot be more tightly bounded by considering other classes of storage functions. }
    \label{fig:reset_resampled}
\end{figure}

\section{Conclusions}
In this paper, we have presented a method for bounding scaled graphs of reset systems. Our approach builds on connections between LMIs, IQCs, dissipativity, and SGs using piecewise quadratic storage functions. The method presented is applied to example reset systems and is able to achieve  significantly tighter bounds.  Furthermore, we have presented a method to investigate the conservatism in the over-bound by exploiting inherent limitations of the quadratic dissipativity framework.  Using these limitations, we can construct a minimal achievable over-bound based on samples.  In an example, we show that the current over-bound is close to the best that can be obtained using the framework based on quadratic dissipativity. The tightness of the current over-bound is still unknown and is left for future work.


\small
\bibliography{Sources}             

@ARTICLE{HosseinFrequency,
  author={van Eijk, Luke F. and Kostić, Dragan and Hassan HosseinNia, S.},
  journal={IEEE Control Systems Letters}, 
  title={Frequency Response Analysis of General Zero-Crossing Reset Control Systems}, 
  year={2025},
  volume={9},
  number={},
  pages={1105-1110},
}

@misc{krebbekx_graphical_2025,
	title = {Graphical {Analysis} of {Nonlinear} {Multivariable} {Feedback} {Systems}},
	publisher = {arXiv},
	author = {Krebbekx, Julius P. J. and Tóth, Roland and Das, Amritam},
	month = jul,
	year = {2025},
	note = {arXiv:2507.16513 [eess]},
}

@misc{krebbekx_reset_2025,
	title = {Reset {Controller} {Analysis} and {Design} for {Unstable} {Linear} {Plants} using {Scaled} {Relative} {Graphs}},
	publisher = {arXiv},
	author = {Krebbekx, Julius P. J. and Tóth, Roland and Das, Amritam},
	month = jun,
	year = {2025},
	note = {arXiv:2506.13518 [eess]},
}

@article{beker_plant_2001,
	title = {Plant with integrator: an example of reset control overcoming limitations of linear feedback},
	volume = {46},
	journal = {IEEE Trans. Autom. Control},
	author = {Beker, O. and Hollot, C.V. and Chait, Y.},
	month = nov,
	year = {2001},
	pages = {1797--1799},
}

@misc{groot_dissipativity_2025,
	title = {A {Dissipativity} {Framework} for {Constructing} {Scaled} {Graphs}},
	publisher = {arXiv},
	author = {de Groot, Timo  and Heemels, Maurice and van den Eijnden, Sebastiaan },
	month = jul,
	year = {2025},
	note = {arXiv:2507.08411 [math]},
}

@article{clegg_nonlinear_1958,
	title = {A nonlinear integrator for servomechanisms},
	volume = {77},
	journal = {Transactions of the American Institute of Electrical Engineers, Part II: Applications and Industry},
	author = {Clegg, J. C.},
	month = mar,
	year = {1958},
	pages = {41--42},
}

@article{zheng_experimental_2000,
	title = {Experimental demonstration of reset control design},
	volume = {8},
	journal = {Control Engineering Practice},
	author = {Zheng, Y. and Chait, Y. and Hollot, C. V. and Steinbuch, M. and Norg, M.},
	month = feb,
	year = {2000},
	pages = {113--120},
}

@article{van_den_eijnden_scaled_2024,
	series = {2024 {European} {Control} {Conference} {Special} {Issue}},
	title = {Scaled graphs for reset control system analysis},
	volume = {80},
	issn = {0947-3580},
	journal = {European Journal of Control},
	author = {van den Eijnden, S. and Chaffey, T. and Oomen, T. and Heemels, W. P. M. H.},
	month = nov,
	year = {2024},
	pages = {101050},
}

@inproceedings{zaccarian_first_2005,
	address = {Portland, OR, USA},
	title = {First order reset elements and the {Clegg} integrator revisited},
	isbn = {978-0-7803-9098-0},
	language = {en},
	booktitle = {Proceedings of the 2005, {American} {Control} {Conference}, 2005.},
	publisher = {IEEE},
	author = {Zaccarian, L. and Nesic, D. and Teel, A. R.},
	year = {2005},
	pages = {563--568},
}

@article{chait_horowitzs_2002,
	title = {On {Horowitz}'s contributions to reset control},
	volume = {12},
	journal = {Int. J. Robust Nonlinear Control},
	author = {Chait, Yossi and Hollot, C. V.},
	year = {2002},
	pages = {335--355},
}

@misc{nauta2025computable,
    title={Computable Characterisations of Scaled Relative Graphs of Closed Operators}, 
    author={Talitha Nauta and Richard Pates},
    year={2025},
    eprint={2511.08420},
    archivePrefix={arXiv},
    primaryClass={eess.SY},
    note = {arXiv:2511.08420 [eess.SY]},
}

@article{nesic_stability_2011,
	title = {Stability and {Performance} of {SISO} {Control} {Systems} {With} {First}-{Order} {Reset} {Elements}},
	volume = {56},
	journal = {IEEE Trans. Autom. Control},
	author = {Nesic, Dragan and Teel, Andrew R. and Zaccarian, Luca},
	month = nov,
	year = {2011},
	pages = {2567--2582},
}

@inproceedings{j_macfarlane_multivariable_1976,
	title = {Multivariable {Nyquist}-{Bode} and multivariable {Root}-{Locus} techniques},
	booktitle = {1976 {IEEE} {Conference} on {Decision} and {Control} including the 15th {Symposium} on {Adaptive} {Processes}},
	author = {Macfarlane, J. A. G.},
	month = dec,
	year = {1976},
	pages = {342--347},
}

@inproceedings{degrootExploitingStructureMIMO2025,
	title = {Exploiting {Structure} in {MIMO} {Scaled} {Graph} {Analysis}},
	issn = {2576-2370},
	urldate = {2026-04-17},
	booktitle = {2025 {IEEE} 64th {Conference} on {Decision} and {Control} ({CDC})},
	author = {de Groot, Timo and Oomen, Tom and Van Den Eijnden, Sebastiaan},
	month = dec,
	year = {2025},
	pages = {6517--6522},
}

@article{erlandsen_continuity_2026,
	title = {Continuity {Conditions} for {Piecewise} {Quadratic} {Functions} on {Simplicial} {Conic} {Partitions} are {Equivalent}},
	issn = {1558-2523},
	journal = {IEEE Trans. on Autom. Control},
	author = {Erlandsen, M. J. and Meijer, T. J. and Heemels, W. P. M. H. and Eijnden, S. J. A. M. van den},
	year = {2026},
	pages = {1--8},
}

@article{johanssonComputationPiecewiseQuadratic1998a,
	title = {Computation of piecewise quadratic {Lyapunov} functions for hybrid systems},
	volume = {43},
	issn = {1558-2523},
	number = {4},
	urldate = {2026-05-12},
	journal = {IEEE Trans. on Autom. Control},
	author = {Johansson, M. and Rantzer, A.},
	month = apr,
	year = {1998},
	pages = {555--559},
}

@article{chaffey_graphical_2023,
  title={Graphical nonlinear system analysis},
  author={Chaffey, T. and Forni, F. and Sepulchre, R.},
  journal={IEEE Trans. Autom. Control},
  volume={68},
  number={10},
  pages={6067--6081},
  year={2023},
  publisher={IEEE}
}

@article{ryu_scaled_2022,
  title={Scaled relative graphs: Nonexpansive operators via 2D Euclidean geometry},
  author={Ryu, E. K. and Hannah, R. and Yin, W.},
  journal={Mathematical Programming},
  volume={194},
  number={1},
  pages={569--619},
  year={2022},
  publisher={Springer}
}

@misc{chen2025softhardscaledrelative,
      title={Soft and Hard Scaled Relative Graphs for Nonlinear Feedback Stability}, 
      author={C. Chen and S. Z. Khong and R. Sepulchre},
      year={2025},
      eprint={2504.14407},
      archivePrefix={arXiv},
      primaryClass={eess.SY},
      note = {arXiv:2504.14407 [eess]},
}

@article{aangenent_performance_2010,
	title = {Performance analysis of reset control systems},
	volume = {20},
	issn = {1099-1239},
	number = {11},
	journal = {Int. J. Robust Nonlinear Control},
	author = {Aangenent, W. H. T. M. and Witvoet, G. and Heemels, W. P. M. H. and van de Molengraft, M. J. G. and Steinbuch, M.},
	year = {2010},
	pages = {1213--1233},
}

\end{document}